\documentclass{article}
\usepackage{amsfonts}
\usepackage{amsthm}
\usepackage{amssymb}
\usepackage{amsmath}
\usepackage{times}
\newtheorem{thm}{Theorem}
\newtheorem{lem}{Lemma}
\newtheorem{defn}{Definition}

\newtheorem{cor}{Corollary}

\newcommand \C{\mathcal{C}}
\newcommand \D{\mathcal{D}}

\newcommand \mR{\mathbb{R}}

\newcommand \mN{\mathbb{N}}

\newcommand \V{\varepsilon}

\newcommand{\Spec}{\operatorname{Spec}}
\newcommand{\Gen}{\operatorname{Gen}}
\newcommand{\ev}{\operatorname{ev}}
\newcommand{\supp}{\operatorname{supp}}
\begin{document}

\title{On real-valued homomorphisms in countably generated differential structures}
\author{Micha{\l} Jan Cukrowski, Zbigniew Pasternak-Winiarski, Wies{\l}aw Sasin
\footnote{Faculty of Mathematics and Information Science,  Warsaw
University of Technology, Pl. Politechniki 1, 00-661 Warsaw,
POLAND}}
\maketitle

m.cukrowski@mini.pw.edu.pl, pastwin@mini.pw.edu.pl,
wsasin@mini.pw.edu.pl\\

\begin{abstract}
Real valued homomorphisms on the algebra of smooth functions on a
differential space are described. The concept of generators of
this algebra is emphasized in this description.\bigskip


\end{abstract}

\section{Introduction}
When all the real homomorphisms defined on an algebra of real
functions defined on a space are evaluations then we say that the
space is smoothly real-compact. There are many articles stating this
property for various spaces. In \cite{prusell}, \cite{ercan} it is
shown that the spaces of real continuous functions on $\mR$ and
$\mR^n$ are smoothly real-compact. In \cite{arias} this property has
been shown for the spaces of functions of class $C^k$
$(k=1,\dots,\infty)$ on separable Banach spaces. Much information
about this topic can be found in \cite{gonz}. The most important
from the point of view of Sikorski spaces is the article
\cite{michor} since it discusses smooth real-compactness of smooth
spaces which are a wider category than the Sikorski spaces. Many
conditions for those spaces to be smoothly real-compact are given
there. In \cite{adam} and \cite{kriegl1} many important results were
obtained for a very wide class of algebras. In our article we
emphasize the concept of generators of a differential space. We use
techniques suitable for Sikorski spaces. Real valued homomorphisms
are classified by their values on generators.
\section{Basic concepts and definitions}

Let $M$ be a nonempty set and $\C$ a set of real functions on $M$.
We introduce on $M$ a topology $\tau_\C$, the weakest topology in
which the functions from $\C$ are continuous. We say that the set
$\C$ is \emph{closed with respect to superposition} if all functions
of a form $\omega\circ (f_1,\dots,f_n)$ where $f_1,\dots, f_n\in \C,
\quad \omega\in C^\infty(\mR^n)$, $n\in \mN$, are in $\C$. Adding to
$\C$ all the functions of this form we obtain what we will call the
superposition closure of $\C$, denoted by $sc\C$ following
Waliszewski \cite{wal1}. For any $A\subseteq M$ by $\C_A$ we denote
the set of all functions $f$ on $A$ such that for any $p\in A$ there
exists an open neighborhood $U\in \tau_\C$ of $p$ and a function
$g\in\C$ such that $f|_{U\cap A}=g|_{U\cap A}$. If $\C=\C_M$ then we
say that $\C$ is \emph{closed with respect to localization}
\cite{sik1}. We call the set of real functions $\C$ on a nonempty
set $M$  a
\emph{differential structure} if it is:\\
i) closed with respect to superposition, $\C=sc\C$,\\
ii) closed with respect to localization, $\C=\C_M$.\\
A differential structure is always an algebra with unity and
contains all constant functions.
\begin{defn}
A pair $(M,\C)$ is said to be a differential space if $M$ is a
nonempty set and $\C$ a differential structure on it.
\end{defn}
We define a \emph{differential subspace} of a differential space
$(M,\C)$ to be any pair $(A,\C_A)$ where $A\subseteq M$, $A\neq
\emptyset$.
\begin{defn}
The differential structure $\C$ is generated by a set of functions
$\C_0$ if\\ $\C=(sc\C_0)_M$.
\end{defn}
Thus $\C$ is the smallest differential structure that contains
$\C_0$. Sometimes we write $\C=\Gen\C_0$. If $\C=\Gen\C_0$ then for
any $f\in \C$ and any point $p\in M$ there exists an open
neighborhood  $U\in \tau_{\C}$ of $p$ and functions
$f_1,\dots,f_n\in \C_0,\quad \omega\in C^\infty(\mR^n)$, $n\in \mN$
such that $f|_U=\omega\circ(f_1,\dots,f_n)|_U$. We say that the
differential space $(M,\C)$ is \emph{finitely generated} if it is
generated by a finite set of a real functions. A differential space
is \emph{countably generated} if it is generated by a countable set
of real functions but it is not finitely generated.

We denote by $(\mR^I,\V_I)$ the differential space with the
structure $\V_I$ generated by the set of projections $\C_0=\{\pi_i
:i\in I\}$, where $\pi_i:\mR^I\rightarrow \mR$ is defined by
$\pi_i(x)=x_i$ for $x=(x_i)_{i\in I}$. This is a generalization of
the Cartesian space $(\mR^n,\V_n)$ where $\V_n=C^\infty(\mR^n)$.

The \emph{spectrum} of an algebra $\C$ is the set
$$\Spec\C=\{\chi:\C\rightarrow \mR\},$$ where $\chi$ is a homomorphism that preserves unity. \\
Evaluation of the algebra $\C$ at a point $p\in M$ is the
homomorphism $\chi\in \Spec\C$ given by
\begin{equation}
\chi(f)=f(p) \quad \forall f\in \C.
\end{equation}
We will denote it by $\ev_p$. We define the mapping\\
$\ev:M\rightarrow \Spec\C$  by the formula:
\begin{equation}
\ev(p)=\ev_p.
\end{equation}
\begin{defn}
(\cite{michor}) We say that a differential space $(M,\C)$ is
smoothly real-compact if any $\chi\in \Spec\C$ is evaluation at some
point $p\in M$.
\end{defn}
From this definition it follows that the space $(M,\C)$ is smoothly
real-compact  when the mapping $\ev$ is a surjection. For any $f\in
\C$ we define the function $\hat{f}:\Spec\C\rightarrow \mR$ by
\begin{equation}
\hat{f}(\chi)=\chi(f) \quad \forall \chi\in \Spec\C.
\end{equation}
The set of all functions of the form $\hat{f}$ will be denoted by
$\hat{\C}$. Define $\tau:\C\rightarrow \hat\C$ by
\begin{equation}
\tau(f)=\hat{f}\quad \forall f\in \C.
\end{equation}
The mapping $\tau$ is an isomorphism between the algebra $\C$ and
the algebra $\hat{\C}$.
\section{Main results}
\begin{lem} The differential space $(\mR^n,\V_n)$ is smoothly real-compact.\label{lem:rn}
\end{lem}
\begin{proof}
Let $\chi\in \Spec\V_n$. We define $p\in \mR^n$ by
$p_i:=\chi(\pi_i)$ for $i=1,\dots,n$. We will show that $\chi=ev_p$.
It is known that any $f\in \V_n$ can be represented as
\begin{equation}
f=f(p)+\sum_{i=1}^ng_i(\pi_i-p_i)\quad \text{for} \quad
g_1,\dots,g_n\in \varepsilon_n,
\end{equation}
where the functions $g_i$ satisfy $g_i(p)=\partial_if(p)$. Then
\[\chi(f)=\chi(f(p))+\sum_{i=1}^n\chi(g_i)(\chi(\pi_i)-\chi(p_i))=f(p)+\sum_{i=1}^n\chi(g_i)(p_i-p_i)=f(p)\].
Therefore $\chi(f)=f(p)$ for all $f\in \V_n$.
\end{proof}
Now we prove:
\begin{lem}
Every differential subspace of the differential space $(\mR^n,\V_n)$
is smoothly real-compact.\label{lem:findimsub}
\end{lem}
\begin{proof}
Let $(M,\C)$ be a differential subspace of $(\mR^n,\V_n)$. The
inclusion mapping $\iota_M:M\rightarrow \mR^n$ is smooth and
therefore $\iota_M^*:\V_n\rightarrow M$ is a homomorphism. From the
definition we know that $\iota_M^*(f)=f|_M$ for all $f\in \V_n$. For
any $\chi \in \Spec\C$ we have $\chi\circ\iota_M^*\in \Spec\V_n$.
From Lemma \ref{lem:rn} we know that there exists $p\in \mR^n$ such
that $\chi\circ\iota_M^*(f)=\chi(\iota_M^*(f))=\chi(f|_M)=f(p)$ for
all $f\in \V_n$. Suppose that $p\notin M$. Define $\omega\in \V_n$
by
\begin{equation}
\omega(x_1,\dots,x_n)=(x_1-p_1)^2+\dots+(x_n-p_n)^2.
\end{equation}
Since $\omega|_M>0$ we have $\frac{1}{\omega|_M}\in \C$. We also
know that $\chi((\omega|_M)(\frac{1}{\omega|_M}))=\chi(1)=1$, and
$(\chi\circ\iota_M^*)(\omega)=\chi(\omega|M)=\omega(p)=0$. This is a
contradiction.

We will show that $\chi=ev_p$. Let $f\in \C$. There exists an open
neighborhood $U\in \tau_{\V_n}$ of  $p$ and a function $\kappa\in
\V_n$ such that $f|_{U\cap M}=\kappa|_{U\cap M}$. From \cite{sik} we
know that there exists a bump function $\phi\in \V_n$ with
$\phi(p)=1$, $\phi|_{M\cap U}>0$ and  $\phi|_{\mR^n-(M\cap U)}=0$.
From these properties it follows that $(f-\kappa|_M)\phi|_M=0$. Then
$\chi((f-\kappa|_M)\phi|_M)=(\chi(f)-\chi(\kappa|_M))\chi(\phi|_M)=0$.
But $\chi(\phi|_M)=(\iota_M\circ\chi)(\phi)=\phi(p)=1$ so
$\chi(f)=\chi(\kappa|_M)=\kappa(p)=f(p)$. We have shown that
$\chi(f)=f(p)$ for all $f\in\C$.
\end{proof}
If the differential structure $\C$ of the differential space
$(M,\C)$ is generated by a set of functions $\C_0$ then we can
define a mapping $\phi:M\rightarrow \mR^{\C_0}$ by
\begin{equation}
\phi(p)(f)=f(p), \quad f\in \C_0.
\end{equation}
We will call this mapping the \emph{generator embedding}. We can
prove the following:
\begin{lem}
A differential space $(M,\C)$ with $\C=\Gen\C_0$ is smoothly
real-compact iff the differential space $(\phi(M),(\V_I)_{\phi(M)})$
for $I=|\C_0|$ is smoothly real-compact.\label{lem:embed}
\end{lem}
\begin{proof}
If $\C_0$ separates the points of $M$ then $\phi$ is a
diffeomorphism onto its image so the result is obvious. So assume
that $\C_0$ does not separate points. Then $\bar\phi:M\rightarrow
\phi(M)$ where $\bar\phi(p)=\phi(p)$ is surjective but not
injective. Set $F:=\bar\phi$. We know that
$F^*:(\V_I)_{\phi(M)}\rightarrow \C$ is an isomorphism of algebras.
If $(M,\C)$ is smoothly real-compact then for any $\nu\in
\Spec(\V_I)_{\phi(M)}$ there exists $\mu\in \Spec\C$ such that
$\mu=\nu\circ(F^*)^{-1}$. Then for any $g\in (\V_I)_{\phi(M)}$,
$\nu(g)=\mu(F^*(g))=\mu(g\circ F)=g(F(p))$. So if $\mu=ev_p$ then
$\nu=ev_{F(p)}$.

If $(\phi(M),(\V_I)_{\phi(M)})$ is smoothly real-compact then for
any $\mu\in \Spec\C$ there exists $\nu\in \Spec(\V_I)_{\phi(M)}$
defined by $\nu=\mu\circ F^*$, so $\mu=\nu\circ (F^*)^{-1}$.
Therefore for any $f\in \C$ we have $\mu(f)=(\nu\circ
(\phi^*)^{-1})(f)=\nu((\phi^*)^{-1}(f))=((\phi^*)^{-1}(f))(q)=f(p)$
for any $p\in F^{-1}(q)$. So if $\nu=ev_q$ then $\mu=ev_p$ for all
$p\in F^{-1}(q)$.
\end{proof}
From the last lemma we know that it is sufficient to work on
subspaces of  Cartesian spaces.
\begin{cor}
Let $(M,\C)$ be a differential space with $\C=\Gen\C_0$ for some
finite $\C_0$. Then $(M,\C)$ is smoothly
real-compact.\label{cor:sub}
\end{cor}
\begin{proof}
By using the generators $\C_0$ we can embed $(M,\C)$ into
$(\mR^{\C_0},(\V_{\C_0})_{\phi(M)})$ and then from Lemmas
\ref{lem:findimsub}, \ref{lem:embed} we derive that $(M,\C)$ is
smoothly real-compact.
\end{proof}
From Corollary \ref{cor:sub} we obtain:
\begin{lem}
Let $(M,\C)$ be a differential space. Any $\chi\in \Spec\C$
satisfies the following condition:
\begin{equation}
\chi(\omega\circ (f_1,\dots,f_n))=\omega(\chi(f_1),\dots,\chi(f_n)),
\end{equation}
for all  $\omega\in \V_n$ and $f_1,\dots,f_n\in \C$, $n\in
\mN$.\label{lem:comp}
\end{lem}
\begin{proof}
Let $\beta_1,\dots,\beta_n\in \C$ be arbitrary functions. We define
the mapping \\$F:(M,\C)\rightarrow (\mR^n,\V_n)$ by:
$$
F(p)=(\beta_1(p),\dots,\beta_n(p)),\quad p\in M.
$$
This mapping is smooth and it is onto its image. Therefore the
mapping\\ $F^*:(\V_n)_{F(M)}\rightarrow \C$ is a homomorphism. For
any $\chi\in \Spec\C$ we have $\chi\circ F^*\in
\Spec((\V_n)_{F(M)})$. From  Corollary \ref{cor:sub} we know that
there exists $q\in F(M)$ such that $\chi\circ F^*=ev_q$ for some
$q\in F(M)$. Also there exists $p\in M$ such that
$$
(\chi\circ F^*)(\omega|_{F(M)})=ev_{F(p)}(\omega|_{F(M)}) \quad
\forall \omega\in \V_n.
$$
We can rewrite this in the form
$$
\chi(\omega\circ
F)=\omega(F(p))=\omega(\beta_1(p),\dots,\beta_n(p))\quad \forall
\omega \in \V_n. $$ By setting $\omega=\pi_i$, $i=1,\dots,n,$ we
obtain $\chi(\beta_i)=\chi(\pi_i\circ F)=\pi_i(F(p))=\beta_i(p)$ and
finally $\chi(\omega\circ
(\beta_1,\dots,\beta_n))=\omega(\chi(\beta_1),\dots,\chi(\beta_n))$
for all  $\omega\in \V_n$.
\end{proof}
Now we prove the following:
\begin{lem}
Let $(M,\C)$ be a differential space such that $\C=\Gen\C_0$ and let
$\chi\in \Spec\C$. If $\chi|_{\C_0}=ev_p|_{\C_0}$ then
$\chi=ev_p$.\label{lem:5}
\end{lem}
\begin{proof}
First we will show that if $f\in sc\C_0$ then $\chi(f)=f(p)$ . From
Lemma \ref{lem:comp} we know that
$\chi(\omega\circ(\beta_1,\dots,\beta_n))=\omega(\chi(\beta_1),\dots,
\chi(\beta_n))$ for $\omega\in \V_n$ and $\beta_1,\dots,\beta_n\in
\C_0$. We also know that $\chi(\beta_i)=ev_p(\beta_i)=\beta_i(p)$.
We can write
$\chi(\omega\circ(\beta_1,\dots,\beta_n))=\omega(\beta_1(p),\dots,\beta_n(p))=\omega\circ(\beta_1,\dots,\beta_n)(p)=ev_p(\omega\circ
(\beta_1,\dots,\beta_n))$. So we see that
$\chi|_{sc\C_0}=ev_p|_{sc\C_0}$.

Now let $f\in \C$ be an arbitrary function. We know that for every
$p\in M$ there exists an open neighborhood $U\in \tau_\C$, functions
$\beta_1,\dots,\beta_n\in \C_0$ and a function $\omega\in \V_n$ such
that $f|_U=\omega\circ(\beta_1,\dots,\beta_n)|_U$. There also exists
a bump function $\psi$ which separates the point $p$ in the set $U$.
This function is constructed by compositing some function from
$\V_n$ with some generators from $\C_0$. We know that the
homomorphism $\chi$ equals evaluation at $p$ on this function, so
$\chi(\phi)=\phi(p)=1$. Now the following equality holds:
$\phi\cdot(f-\omega\circ(\beta_1,\dots,\beta_n))=0$. By applying the
homomorphism $\chi$ to this equality we obtain
$\chi(\phi)\cdot\chi(f-\omega\circ(\beta_1,\dots,\beta_n))=\chi(f)-\chi(\omega\circ(\beta_1,\dots,\beta_n))=0$
so
$\chi(f)=\chi(\omega\circ(\beta_1,\dots,\beta_n))=ev_p(\omega\circ(\beta_1,\dots,\beta_n))=f(p)=ev_p(f)$.
We see that $\chi(f)=f(p)$ for all $f\in \C$.
\end{proof}
From Lemma \ref{lem:5} we get:
\begin{cor}
The differential space $(\mR^I,\V_I)$ is smoothly real-compact.
\end{cor}
\begin{proof}
Let $\chi\in \Spec\V_I$ be any homomorphism. Define $p\in \mR^I$ by
 $p_i=\chi(\pi_i)$ for $i\in I$. Then $\chi(\pi_i)=\pi_i(p)$ so
$\chi(\pi_i)=ev_p(\pi_i)$. Since the structure $\V_I$ is generated
by the set $\{\pi_i:i\in I\}$ we see that $\chi$ is evaluation at
$p$ on the generators. From the last Lemma we derive that $\chi$ is
an evaluation on the whole $\V_I$.
\end{proof}

By using the whole $\C$ as the set of generators we can embed $M$ in
$\mR^\C$. We denote this embedding by $\iota$ so $\iota:M\rightarrow
\mR^\C$, $\iota(p)_f=f(p)$. This is a special case of a generator
embedding. We can also map $\Spec\C$ into $\mR^\C$ using the mapping
$\kappa:\Spec\C\rightarrow \mR^\C$ defined by
$\kappa(\chi)_f=\hat{f}(\chi)=\chi(f)$. It is obvious that
$\iota=\kappa\circ ev$. In \cite{michor} Kriegl, Michor and
Schachermayer have shown that $\iota(M)$ is dense in
$\kappa(\Spec\C)$ in the Tikhonov topology of $\mR^\C$. Since the
mapping $\kappa$ is a homeomorphism one can easily see:
\begin{cor}
$ev(M)$ is dense in $\Spec\C$ in the topology
$\tau_{\hat{\C}}$.\label{cor:dense}
\end{cor}

This property will allow us to prove an interesting fact about the
space $(\Spec\C,\hat{\C})$.
\begin{lem} If $(M,\C)$ is a differential space
then $(\Spec\C,\hat\C)$ is a differential space.
\end{lem}
\begin{proof}
To prove that $(\Spec\C,\hat{\C})$ is a differential space, we have
to show that the set $\hat{\C}$ is closed with respect to
superposition with smooth functions from $\V_n$ and is closed with
respect to localization.

Let $g=\omega\circ (\hat{f_1},\dots,\hat{f_n})$ for some $\omega\in
\V_n$ and $\hat{f_1},\dots,\hat{f_n}\in \hat{\C}$. From Lemma
\ref{lem:comp} we know that\\
$g(\chi)=\omega\circ(\hat{f_1},\dots,\hat{f_n})(\chi)=\tau(\omega\circ
(f_1,\dots,f_n))(\chi)\quad \forall \chi\in \Spec\C$. We have shown
that $g\in \hat{\C}$ so $\hat\C$ is closed with respect to
superposition.

Let a function $f:\Spec\C\rightarrow \mR$ satisfy the localization
condition in the space $(\Spec\C,\hat{\C})$. For any open subset
$\hat{U}\in \Spec\C$ there is $\hat{g}\in \hat{\C}$ such that
$f|_{\hat{U}}=\hat{g}|_{\hat{U}}$. We can uniquely define a function
$h:M\rightarrow \mR$ by the condition $h(p)=f(ev_p)$ for all $p\in
M$. For any open set $\hat{U}\in \Spec\C$ the set  $U=\{p\in M :
ev_p\in \hat{U}\}$ is open. From the definitions of  $h$ and $U$ we
know that $h|_U=g|_U$. Because $g\in \C$ it follows that $h\in \C$.
We also know that $\hat{h}|_{evM}=f|_{evM}$. From Corollary
\ref{cor:dense} we derive that $f=\hat{h}$. This means that $f\in
\hat{\C}$ so $\hat{\C}$ is closed with respect to localization.
\end{proof}
Now one can prove the following lemmas:
\begin{lem}
If  $(M,\C)$ is a differential space with the structure $\C$
generated by $\C_0$ then the differential structure $\hat{\C}$ of
the differential space $(\Spec\C,\hat\C)$ is generated by
$\hat{\C_0}$.
\end{lem}
\begin{proof}
Assume that $\C_0=\{f_i:i\in I\}$. We know that for any $f\in \C$
there exists an open covering of $M$ such that on each set $U$ of
this covering the function $f$ can be expressed in the form
$\omega\circ (f_1,\dots,f_n)$ where $f_1,\dots, f_n\in \C$ and
$\omega \in \V_n$. For each open set $U$ of the covering we define
$\hat{U}=\{ev_p\in \Spec\C: p\in U\}$. On the set $\hat{U}$ we
consider the function $\hat{f}=\tau(\omega\circ (f_1,\dots,f_n))$.
The sets of the form $\hat{U}$ might not be a covering of $\Spec\C$
but thair union is dense in $\Spec\C$. Therefore we can prolong
uniquely this representation of $\hat{f}$ on the whole $\Spec\C$. We
have shown that $\hat{\C}=\Gen\hat{\C_0}$.
\end{proof}
\begin{lem}
For any differential space $(M,\C)$ the differential space
$(\Spec\C,\hat{\C})$ is smoothly real-compact.
\end{lem}
\begin{proof}
We need to show that for every homomorphism $\hat\chi\in
\Spec\hat{\C}$ there  exists a homomorphism $\psi\in \Spec\C$ such
that $\hat\chi=ev_\psi$. Since the algebras $\C$ and $\hat\C$ are
isomorphic we can define uniquely $\chi\in \Spec\C$ by the formula
$\chi(f)=\hat\chi(\hat{f})$. We will show that $\hat\chi=ev_\chi$.
Indeed $ev_\chi(\hat{f})=\hat{f}(\chi)=\chi(f)=\hat\chi(\hat{f})$.
\end{proof}

\begin{lem}\label{lem:uniq}
Let $(M,\C)$ be a differential space and $\C=\Gen\C_0$. If
$\chi_1,\chi_2\in \Spec\C$ are equal on the generators,
$\chi_1|_{\C_0}=\chi_2|_{\C_0}$, then they are equal,
$\chi_1=\chi_2$.
\end{lem}
\begin{proof}
Assume that $\chi_1|_{\C_0}=\chi_2|_{\C_0}$ and $\chi_1\neq\chi_2$.
From the last lemma we know that the differential structure
$\hat{\C}$ of the differential space $(\Spec\C,\hat{\C})$ is
generated by $\hat{\C}_0$. From the condition
$\chi_1|_{\C_0}=\chi_2|_{\C_0}$ we derive that
$\hat{f}(\chi_1)=\hat{f}(\chi_2)$, for all $\hat{f}\in \hat{\C}_0$.
But we know that if the generators do not separate points then all
the functions do not separate points so $\forall \hat{f}\in \hat{\C}
\hat{f}(\chi_1)=\hat{f}(\chi_2)$ and it follows that
$\chi_1(f)=\chi_2(f)$for all $f\in \C$ . This means that
$\chi_1=\chi_2$.
\end{proof}
\begin{lem}
If $(M,\C)$ is differential subspace of  $(\mR^I,\V_I)$ then any
function $f\in \C$ is uniquely continuously prolongable to
$\tilde{f}:\tilde{M}\rightarrow \mR$, where $\tilde{M}=\{p\in \mR^I
: \exists \chi\in \Spec\C$ such that
 $p_i=\chi(\pi_i|_M) \quad\forall i \in I\}$.\label{lem:prelong}
\end{lem}
\begin{proof}
We define $\tilde{f}(p)=\hat{f}(\chi)$ where $\chi\in \Spec\C$ is
such that $\chi(\pi_i)=p_i$ for all $i\in I$. Since a homomorphism
is uniquely defined by its value on the generators (Lemma
\ref{lem:uniq}) this definition is correct. We see that if $p\in M$
then $\chi=ev_p$ and $\tilde{f}(p)=\hat{f}(ev_p)=f(p)$ so this is
indeed a prolongation. This prolongation is continuous since the
function $\tilde{f}$ is a realization of the function $\hat{f}$ on
the set $\tilde{M}$ which is the image of  $\Spec\C$ under the
generator embedding using the generators $\tau(\pi_i|_M), i\in I$.
Uniqueness follows from the fact that $M$ is dense in $\tilde{M}$ in
the topology of $\mR^I$.
\end{proof}
From Lemma \ref{lem:prelong} we obtain:
\begin{cor}\label{cor:prelong}
When $(M,\C)$ is a differential subspace of $(\mR^I,\V_I)$ generated
by  $\C_0=\{\pi_i|_M: i\in I\}$ then the mapping
$\chi:\C_0\rightarrow\mR$ defined on generators
by$\chi(\pi_i|_M)=p_i$ for some $p\in\tilde{M}-M$ can be prolonged
to a homomorphism on the whole $\C$ iff all the functions from $\C$
are prolongable to $p$.
\end{cor}

Let $M=\mR^\mN-\{0\}$ and $\C_M=(\V_\mN)_M$. Then $(M,\C_M)$ is a
differential subspace of $(\mR^\mN,\V_\mN)$. We will show that this
space is smoothly real-compact.
\begin{lem}\label{lem:prelong1}
There exists a function $\xi\in \C_M$ which is not prolongable to
any continuous function on $\mR^\mN$.
\end{lem}
\begin{proof}
We know that there exists a function $\phi\in C^\infty(\mR)$
satisfying the
following properties:\\
1. $\forall x\in \mR\quad \phi(x)\in [0,1]$\\
2. $supp(\phi)\in (-\infty,1]$\\
3. $\phi|_{[0,\frac{1}{2}]}=1$\\
For any $k\in \mN$ we  define $\tilde\rho_k:\mR^\mN\rightarrow \mR$
by the formula:
$$
\tilde{\rho}_k((x_n))=\sum_{i=1}^k x_i^2
$$
for $(x_n)\in \mR^\mN$. Then $\tilde{\rho}_k\in
\C^{\infty}(\mR^{\mN})$, and  $\rho_k=\tilde{\rho}_k|_M \C_M$. We
define $\xi:M\rightarrow \mR$ by:
\begin{equation}\label{def:f}
\xi((x_n))=\sum_{k=1}^\infty \phi(k^2\rho_k((x_n))).
\end{equation}
We will show that this function belongs to the structure $\C_M$. For
any $k\in \mN$ we can define the closed subset $A_k=\{(x_n)\in M:
k^2\rho_k((x_n))\leq 1\}=\{(x_n)\in M:\rho_k((x_n))\leq
\frac{1}{k^2}\}$. We see that $\supp(\phi\circ(k^2\rho_k))\subseteq
A_k$. For any $(x_n)\in M$ the sequence $\rho_k((x_n))$ is
non-decreasing with respect to $k$ and there exists $k_0\in \mN$ for
which $\frac{1}{k^2}<\rho_{k_0}((x_n))$. This means that
$(x_n)\notin A_k$. Therefore $\bigcap_{k\in \mN} A_k=\emptyset$. We
also know that $A_{k+1}\subseteq A_k$. Let us define the family of
open subsets $U_k=M-A_k$. Of course $\bigcup_{k\in \mN}U_k=M$. If
$(x_n)\in U_k$ then $\phi(k^2\rho_k((x_n))=0$. Then for all $m>k$,
$x_n\in U_m$ so $\phi(m^2\rho_m((x_n)))=0$. This means that only a
finite number of elements are non-zero in the sum (\ref{def:f}) and
therefore
\begin{displaymath}
\xi((x_n))=\sum_{j=1}^{k-1}\phi(j^2\rho_j((x_n))),
\end{displaymath}
so $\xi|_{U_k}\in \C_{U_k}=(\C_M)_{U_k}$ for all $k\in \mN$. From
the localization closedness of the differential structure we derive
that $\xi\in \C_M$. Now we will define a sequence in $M$ convergent
to $0$ on which the function $\xi$  diverges. Let $z_k=(x_{n,k})$
where
\begin{displaymath}
x_{n,k}= \left\{
\begin{array}{lll}
\frac{1}{k\sqrt{2}} & \textrm{for} & n=k, \\
0 & \textrm{for} & n\neq k.
\end{array}\right.
\end{displaymath}
We can see that $ \lim_{k\rightarrow \infty}z_k=0\in \mR^{\mN} $ and
\begin{displaymath}
\rho_j((z_k))= \left\{
\begin{array}{lll}
\frac{1}{2k^2} & \textrm{for} & j\geq k, \\
0 & \textrm{for} & j<k.
\end{array}\right.
\end{displaymath}
For $j\leq k$ we obtain $\phi(j^2\rho_j((x_k)))=1$ and therefore
\begin{displaymath}
\xi((x_k))=\sum_{j=1}^{\infty}\phi(j^2\rho_j((x_k)))\geq\sum_{j=1}^k
1=k.
\end{displaymath}
This means that $\lim_{k\rightarrow{\infty}}\xi((x_k))=+\infty$. The
function $\xi$ is not prolongable to any continuous function in
$\mR^\mN$.
\end{proof}
Now we prove
\begin{lem}
The differential space $(M,\C_M)$ is smoothly real-compact.
\end{lem}
\begin{proof}
From Lemma \ref{lem:uniq} we know that the set $\Spec\C_M$ may
contain only one homomorphism $\chi_0$ which is not an evaluation.
This homomorphism would be defined on the generators by the formula
$\chi_0(\pi_i|_M)=0$ for all $i\in I$. So there would be only one
point $0\in \tilde{M}-M$. But it cannot be so since from Corollary
\ref{cor:prelong} we know that all the functions from $\C_M$ are
prolongable to the point $0$. From the last lemma we know that there
exists a function $\xi\in \C_M$ which is not prolongable.
\end{proof}
One can easily see
\begin{cor}
For any $p\in \mR^\mN$ the differential space
$(\mR^\mN-\{p\},(\V_\mN)_{\mR^\mN-\{p\}})$ is smoothly real-compact.
\end{cor}
\begin{proof}
This space is diffeomorphic to  $(M,\C_M)$ so it is be smoothly
real-compact.
\end{proof}
\begin{defn}
The disjoint union of differential spaces $(M,\C)$ and $(N,\D)$
where $M\cap N=\emptyset$ is the differential space $(M\cup
N,\C\oplus\D)$. The structure $\C\oplus \D$ is defined by the
property $f\in\C\oplus\D \iff f|_M\in \C$ and $f|_N\in \D$.
\end{defn}
We will prove:
\begin{lem}
If differential spaces $(M,\C)$ and $(N,\D)$ are smoothly
real-compact then the differential space $(M\cup N,\C\oplus \D)$ is
smoothly real-compact.
\end{lem}
\begin{proof}
Elements of the algebra $\C\oplus \D$ are pairs $(f,g)$ where $f\in
\C$ and $g\in \D$. Let $\chi\in \Spec(\C\oplus\D)$. We shall show
that it is  evaluation at some point $p\in M\cup N$. From the
equalities $(0,1)+(1,0)=(1,1)$ and $(0,1)(1,0)=(0,0)$ we obtain two
cases:\\
\begin{eqnarray}
1)\quad \chi((1,0))=1 \quad and \quad \chi((0,1))=0\nonumber\\
2)\quad \chi((1,0))=0\quad and  \quad \chi((0,1))=1
\nonumber\end{eqnarray} Since every function from $\C\oplus \D$ can
be uniquely decomposed as $(f,g)=(f,0)(1,0)+(0,g)(0,1)$ the
homomorphism $\chi$ acts as
follows:\\
$\chi((f,g))=\chi((f,0))\chi((1,0))+\chi((0,g))\chi((0,1))$. In case
1) we will get \\ $\chi((f,g))=\chi((f,0))$ and in case 2),
$\chi((f,g))=\chi((0,f))$.\\ The algebra of functions of the form
$((f,0))\in \C\oplus\D$ is isomorphic to $\C$. Therefore a
homomorphisms  $\psi\in \Spec\C$ can be extended to a homomorphism
from $\C\oplus\D$ by the formula $\bar\psi((f,g))=\psi(f)$. All the
homomorphisms in case 1) are of this form. Therefore in case 1) the
homomorphism $\chi((f,g))=\psi(f)$ where $\psi\in \Spec\C$ is such
that $\bar\psi=\chi$. But since the space $(M,\C)$ is smoothly
real-compact there exists a point $p\in M$ such that $\psi=ev_p$.
Then we can write $\chi((f,g))=ev_p((f,g))=(f,g)(p)=f(p)+g(p)$ for
$p\in M\cup N$. We have shown that in case 1) the homomorphism
$\chi$ is an evaluation. For  case 2) the proof is analogous.
\end{proof}

\begin{defn}
We denote by $\tilde{\V}$  the differential structure on $\mR^\mN$
generated by the set $\C_0=\{\pi_i: i\in \mN\}\cup \{\theta_p\}$,
where $\theta_p$ is the characteristic function of the point $p\in
M$.
\end{defn}
\begin{lem}
The differential space $(\mR^\mN,\tilde{\V})$ is smoothly
real-compact.
\end{lem}
\begin{proof}
We can decompose the space $(\mR^\mN,\tilde{\V})$ into the direct
sum of the spaces $(\mR^\mN-\{p\},(\V_\mN)_{\mR^\mN-\{p\}})$ and
$(\{p\}, F(p))$ where $F(p)$ is the algebra of all functions on the
singleton space. From the definition of $(\mR^\mN,\tilde{\V})$ it is
obvious that $\mR^\mN=\{p\}\cup (\mR^\mN-\{p\})$ and
$\tilde{\V}=(\V_\mN)_{\mR^\mN-\{p\}}\oplus F(p)$. Both spaces in the
direct sum are smoothly real-compact so the space
$(\mR^\mN,\tilde{\V})$ is smoothly real-compact.
\end{proof}
Now we present the main results:
\begin{thm}
Any differential subspace of $(\mR^\mN,\V_\mN)$ is smoothly
real-compact.\label{thm:sub}
\end{thm}
\begin{proof}
Let $\iota_M:(M,\C)\rightarrow (\mR^\mN,\V_\mN)$ be the inclusion
mapping. For any $\chi\in \Spec\C$,  $\chi\circ\iota_M^*\in
\Spec(\V_\mN)$. The space $(\mR^\mN,\V_\mN)$ is smoothly
real-compact so there is $p\in \mR^\mN$ such that
$\chi\circ\iota_M^*=ev_p|_{\V_\mN}$.

We need to show that $p\in M$. Assume that $p\notin M$. We can treat
the space $(M,\C)$ as a differential subspace of
$(\mR^\mN,\tilde{\V})$. Let $\nu_M:(M,\C)\rightarrow
(\mR^\mN,\tilde{\V})$ be the inclusion. Then $\chi\circ \nu_M^*\in
\Spec\tilde\V$. Because the space $(\mR^\mN,\tilde{\V})$ is smoothly
real-compact there exists a point $q\in \mR^\mN$ such that
$\chi\circ\nu_M^*=ev_q|_{\tilde{\V}}$. We know that on common
generators $\pi_i$ the equalities $\chi(\pi_i|_M)=ev_p(\pi_i)=p_i$
and $\chi(\pi_i|_M)=ev_q(\pi_i)=q_i$ holds for all $i\in \mN$. This
specifies all the coordinates so $p=q$. Therefore we can write
$\chi\circ\nu_M^*=ev_p|_{\tilde{\V}}$. So
$(\chi\circ\nu_M^*)(\theta_p)=ev_p(\theta_p)=1$. We have a
contradiction with the fact that
$(\chi\circ\nu_M^*)(\theta_p)=\chi(\theta_p|_M)=\chi(0)=0$. We see
that $p\in M$ and $\chi\circ \iota_M^*=ev_p|_{\V_\mN}$. So
$\chi(\pi_i|_M)=ev_p(\pi_i|M)$ for all $i\in \mN$. The set
$\{\pi_i:i\in \mN\}$ is the set of generators of the differential
space $(M,\C)$. We derive that $\chi=ev_p$.
\end{proof}
\begin{cor}
Any countably generated differential space is smoothly real-compact.
\end{cor}
\begin{proof}
Every countably generated differential space can be treated as a
subspace of  $(\mR^\mN,\V_\mN)$. From  Theorem \ref{thm:sub} we know
that all subspaces of this space are smoothly real-compact.
\end{proof}
\section{Conclusion}
We have shown how a real-valued homomorphism act on the algebra of
smooth functions in a differential space. It is sufficient to
examine its value on the generators of the differential structure. A
non-prolongable function on a differential space of sequences
without zero sequence is constructed. From the existence of this
function we deduced that countably generated structural algebra $\C$
of a differential space $(M,\C)$ gives all information about the set
$M$ because the set $\Spec\C$ contains only real valued
homomorphisms of the  form $ev_p$ for $p\in M$. Therefore the
geometry of such spaces can be built on their algebras. The choice
of generators of the differential  structure is not unique so it is
important to choose the smallest one. It is also shown that the pair
$(\Spec\C, \hat{\C})$ is a differential space for any differential
space $(M,\C)$. For countably generated differential spaces,
$(M,\C)$ and $(\Spec\C,\hat{\C})$ are diffeomorphic. Theorem
\ref{thm:sub} has been obtained as a result of observations about
generators. From \cite{adam} and \cite{kriegl1} it also follows as a
conclusion from a much wider theory. In  \cite{ran} there is an
important theorem which gives a sufficient condition for an algebra
of functions on an arbitrary topological space to be smoothly
real-compact. We could easily show that countably generated
differential spaces satisfy this condition. But we have presented
our proof using techniques of differential space theory.

\end{document}